\documentclass[twoside]{amsart}

\usepackage{amssymb}
\usepackage{graphicx}
\usepackage{caption}
\usepackage{bm}
\usepackage{color}

\newcommand{\rmd}{\mathrm{d}}

\newcommand{\diag}{\mathrm{diag}}

\newcommand{\I}{\mathcal{I}}
\newcommand{\J}{\mathcal{J}}

\newcommand{\M}{\phantom{-}}

\newcommand{\mFr}{\mathfrak{m}}

\DeclareMathOperator{\modR}{mod}
\DeclareMathOperator{\Ord}{Ord}

\newcommand{\dd}{\textsf{d}}
\DeclareMathOperator{\SymOtimes}{\overset{\mathrm{sym}}{\otimes}}

\newcommand{\Jac}{\mathrm{Jac}}
\newcommand{\Mat}{\mathrm{Mat}}

\newcommand{\ImR}{\mathrm{Im}\,}

\DeclareMathOperator{\Complex}{\mathbb{C}}

\DeclareMathOperator{\Integer}{\mathbb{Z}}

\newtheorem{lemma}{Lemma}[section]
\newtheorem{theo}{Theorem}[section]

\newtheorem{cor}{Corollary}[section]

\theoremstyle{definition}
\newtheorem{Def}{Definition}
\newtheorem{rem}{Remark}[section]

\theoremstyle{plain}

\title[Igusa's map on vanishing hyperelliptic modular forms]
{Igusa's map on modular forms vanishing on the hyperelliptic locus }
\author{J Bernatska, Y Kopeliovich}
\address{}
\email{jbernatska@gmail.com}
\date{\today}

\allowdisplaybreaks[4]

\begin{document}
 
\maketitle 
\begin{abstract}
We extend Igusa's map $\rho$ to modular forms which vanish 
on the hyperelliptic locus of the Siegel upper half-plane. 
The lowest non-vanishing derivatives of such modular forms are computed
with the help of the general Thomae formula, they serve as vector-valued modular forms.
This approach is illustrated by examples in genera $3$ and $4$.
\end{abstract}


\section{Introduction}

The theory of Siegel modular forms 
 has many applications in various branches of mathematics. 
 It started from attempts to understand the moduli space of a Riemann surface via its period matrix and theta functions. 
One of the most outstanding problems is to produce algebraic equations that would characterize 
the Jacobian locus inside  the Siegel upper-half plane $\mathfrak{S}_g$ of degree $g$. 
This problem was solved in degrees 
$g=4$, $5$ and a significant progress was made for an arbitrary degree $g$, see \cite{FarGrMan}. 
Recently, applications of Siegel modular forms to superstring theory arose, see for example \cite{GrMa}.

Following Igusa's notation, we denote by $\mathfrak{A}(\Gamma_g)$ the graded ring of modular 
forms belonging to the modular 
 group $\Gamma_g = \mathrm{Sp}(g,\Integer)$ of degree (or genus\footnote{We prefer to call $g$ a genus
 in order to distinguish it from other degrees.}) $g$. As shown in \cite{Ig1964},
 $\mathfrak{A}(\Gamma_g)$ is the integral closure of the ring of theta constants
within its field of fractions. 
In \cite{Ig1967},  Igusa defined the homomorphism $\rho$ of $\mathfrak{A}(\Gamma_g)$
to a graded ring $\mathcal{S}(2,2g+2)$ of projective invariants of a binary form of degree $(2g+2)$. 
The homomorphism $\rho$, which we call the map $\rho$ below,
is based on the first Thomae formula, which gives expressions for non-singular theta constants
in terms of roots of  a binary form.  Evidently, the map $\rho$ is defined within the hyperelliptic locus  
of~$\mathfrak{S}_g$,  where the Thomae formulas work. We denote the ring of modular 
forms restricted to the hyperelliptic locus by $\mathfrak{H}(\Gamma_g)$.

The map $\rho$ enabled Igusa to determine the structure of the modular form ring 
in genus $2$, see \cite{Ig1962}. Further, it was shown in \cite{Ig1967}, that in genus $3$ 
an exact sequence exists
$$ 0 \to \chi_{18} \mathfrak{A}(\Gamma_3) \to \mathfrak{A}(\Gamma_3) \overset{\rho}{\to} \mathcal{S}(2,8),$$
where $\chi_{18}$ is the cusp form of weight $18$, defined as  the product of all even theta constants. 
In \cite{Tsu1986}, the behavior of modular forms on $\mathfrak{H}(\Gamma_3)$ 
was completely interpreted as that of the corresponding invariants.
 In the ring  $\mathfrak{H}(\Gamma_3)$ everything is computable rather than in $\mathfrak{A}(\Gamma_3)$. 
 So the ring  $\mathcal{S}(2,8)$  together with the map $\rho$ gives much information about modular forms. 

The map $\rho$ is central in \cite{Poor2008}, 
where the structure of the modular form ring in genus~$4$ is investigated.
Similar ideas are employed in \cite{Poor1996} to show that the generalized Schottky form in 
genus $g$ vanishes 
on the hyperelliptic locus of $\mathfrak{S}_g$. 
Furthermore, some results served as one of motivations to the present paper. 
Namely,  it is shown  in \cite[Cor.\,3.9]{Poor1996}, that a hyperelliptic cusp form  
of  weight less than $8+4/g$  must vanish on the hyperelliptic locus of $\mathfrak{S}_g$.
This implies that further investigation of low weight cusp forms 
requires taking their derivatives. 

The main goal of the present paper is to extend the map $\rho$ 
to modular forms which vanish, but not identically, on the hyperelliptic locus of the Siegel upper half-plane. 
As usual,  the map $\rho$ is considered over the ring $\mathfrak{H}(\Gamma_g)$ 
of theta constants on the hyperelliptic locus.
Until now,  $\rho$ was restricted to non-vanishing modular forms 
from $\mathfrak{H}(\Gamma_g)$, that is modular forms constructed from 
theta constants with non-singular even characteristics.
A recently discovered generalization of the Thomae formulas, see \cite{bGTF},  
allows to work with theta constants vanishing on the hyperelliptic locus.
Using these formulas we are able to compute explicitly the lowest non-vanishing theta
derivatives with singular characteristics. In such a way an extension of the map $\rho$ is accomplished. 
We illustrate  our approach by examples in genera $3$ and $4$.

As a part of our investigation we are led to the question of the vanishing property 
of various modular forms on the hyperelliptic locus. 
So, we need an effective way to construct vanishing modular forms.
We generate modular forms from monomials in theta constants composed
on the base of G\"{o}pel systems, --- that appears to be a fruitful idea. 
Such an approach led us to a beautiful result.
In genus $g=3$, we discovered a pre-image of a homogeneous polynomial $I_1$ 
in roots of a binary $8$-form, which we call
 a quasi-invariant $I_1$, since it is an invariant  of degree $1$ and weight $4$
under unimodular transformations, but not symmetric with respect to permutations of roots. 
Such a quasi-invariant did not appear in the literature.

The paper is organized as follows. In Preliminaries, we briefly describe 
the hyperelliptic locus of $\mathfrak{S}_g$, the connection between not normalized and normalized
periods, introduce interpretation of characteristics in terms of partitions, which is extremely helpful
for classifying characteristics with respect to multiplicities. 
We briefly recall the notions of
G\"{o}pel groups and G\"{o}pel systems on characteristics, and give the old and new Thomae formulas.
In Section $3$, we define objects which we call  vector-valued hyperelliptic modular forms. 
These objects behave as modular forms on the hyperelliptic locus.
Then we investigate the action of the modular group on theta derivatives
and vector-valued modular forms. 

In Sections $4$ and $5$ we present new results regarding 
vector-valued modular forms in genera $3$ and $4$, respectively. 
In more detail, in Section $4$, the derivative of $\chi_{18}$ is defined and shown to be 
a vector-valued modular  form 
of weight $20$. We introduce a bunch of new hyperelliptic modular forms of weight $4$,
vanishing to the first order on the hyperelliptic locus, whose derivatives are 
vector-valued modular forms of weight~$6$.
We also construct modular forms of weight $0$ vanishing on the hyperelliptic locus,
such that the sum of their derivatives produces the mentioned quasi-invariant $I_1$.
In Section $5$, it is shown that in genus $4$, hyperelliptic modular forms vanishing to the first order
exist in weight $4$, and vanishing to the second order in weight $8$.
Finally, a modular form $\chi_{68}$ of weight $68$ is defined similarly to $\chi_{18}$ in genus $3$.
It is shown that $\chi_{68}$ vanishes to the $10$-th order on the hyperelliptic locus of $\mathfrak{S}_4$, 
and its derivative is a vector-valued modular form of  weight~$88$, 
which produce coefficients of the corresponding  binary form of degree $10$.

\section*{Aknowledgement}
The authors are grateful to Cris Poor  for fruitful discussions which led to the present paper.

\section{Preliminaries}
\subsection{Siegel modular forms}
Let $\mathfrak{S}_g = \{\tau \in \Mat(g\times g) \mid \tau^t = \tau, \ImR \tau >0\}$ 
be the Siegel upper-half plane of degree $g$, and
$\Gamma_g = \mathrm{Sp}(g,\Integer)$ be the homogeneous modular group
of degree $g$ of level $1$.  $\Gamma_g$ acts on $\mathfrak{S}_g$ as
\begin{gather}\label{MGActTau}
\forall \gamma = \begin{pmatrix}
  a & b \\ c & d
 \end{pmatrix}  \in  \Gamma_g \qquad
 \gamma \langle \tau \rangle = (a \tau + b) (c \tau + d)^{-1}.
\end{gather}

\begin{Def}
An analytic function $f: \mathfrak{S}_g \mapsto \Complex$  which satisfies ($\kappa>0$)
\begin{gather}\label{MGActTau}
\forall \gamma \in \Gamma_g  \qquad  
f( \gamma \langle \tau \rangle ) = \big( \det(c \tau +d) \big)^\kappa f(\tau) 
\end{gather}
is called a Siegel modular form 
of degree $g$ and weight $\kappa$, or simply a modular form of weight $\kappa$.
\end{Def}
Modular forms of a fixed weight $\kappa$ form a complex vector space
$\mathfrak{A}_\kappa(\Gamma_g)$, and the graded ring 
$\mathfrak{A}(\Gamma_g) = \oplus_{k>0} \mathfrak{A}_\kappa(\Gamma_g)$
is called the ring of modular forms.

We do not restrict the weight $\kappa$ to integers as in \cite{Ig1964}, 
or even integers as in \cite{Tsu1986}. 
In fact, every theta constant is a modular form of weight $1/2$ belonging to $\Gamma_g(2)$, 
and we work with
half-integer weights. 

Also we extend this definition of modular forms to derivatives with respect to $\tau$,
see Section~\ref{s:ThetaTransf}.

Let $\mathfrak{W}_g  = \{\Omega \in \Mat(2g \times 2g) \mid 
\Omega J \Omega^t = 2 \imath \pi J\}$ be the group of
full period matrices with
the complex structure
\begin{gather*}
J = \left( \begin{array}{cc}  \M 0_g & 1_g \\ -1_g & 0_g
\end{array}  \right),
\end{gather*}
where $1_g$ is the identity matrix of size $g$, and $0_g$ is $g\times g$ zero matrix.
The relation $\Omega J \Omega^t = 2 \imath \pi J$ is known as the Legendre period relation.
Evidently, $\mathfrak{W}_g \sim  \mathrm{Sp}(g,\Complex)$.
Applying the Gauss decomposition to a period matrix $\Omega \in \mathfrak{W}_g$
we obtain 
\begin{gather}\label{OMatr}
\Omega = 
\begin{pmatrix} \eta' & \omega' \\ \eta & \omega \end{pmatrix}
= \begin{pmatrix} 1_g & \omega' \omega^{-1} \\ 0 & 1_g \end{pmatrix}
\begin{pmatrix} \eta' - \omega' \omega^{-1} \eta & 0 \\ 0 & \omega \end{pmatrix}
\begin{pmatrix} 1_g & 0 \\ \omega^{-1}  \eta & 1_g \end{pmatrix}.
\end{gather}
Here $\omega$ and $\omega'$ denote first kind integrals along canonical cycles
$\mathfrak{a}_j$ and $\mathfrak{b}_j$, $j=1$, \ldots, $g$, respectively, and 
$\eta$, $\eta'$ denote second kind integrals along the same  cycles.
The group of upper triangular block matrices represents the Jacobian locus
$\mathfrak{J}_g \subset \mathfrak{S}_g$,
and $\tau_{\J} = \omega' \omega^{-1}$ is a Jacobian point.
$\Gamma_g$ acts on $\mathfrak{W}_g$ as
\begin{gather*} 
\forall \gamma \in  \Gamma_g \qquad
 \gamma \langle \Omega \rangle =  
 \begin{pmatrix} a \eta' + b \eta & a \omega' + b \omega \\ 
c \eta' + d \eta & c \omega' + d \omega \end{pmatrix}.
\end{gather*}
The Gauss decomposition of $ \gamma \langle \Omega \rangle$ gives
the action on $\mathfrak{J}_g$:
\begin{subequations}\label{GammaTau}
\begin{align}
&\gamma \langle \omega \rangle = c \omega' + d \omega = (c \tau_\J + d) \omega, \\
&\gamma \langle \omega' \rangle =  a \omega' + b \omega =   (a \tau_\J + b) \omega, \\
&\gamma \langle \tau_\omega \rangle = (a \omega' + b \omega) (c \omega' + d \omega)^{-1}
= (a \tau_\J + b) (c \tau_\J + d)^{-1},
\end{align}
\end{subequations}
the latter coincide with \eqref{MGActTau}. 

Therefore,
We may consider $\mathfrak{J}_g$ as a factor-group
$\mathfrak{W}_g / \mathfrak{P}_g$, where
\begin{gather*}
\mathfrak{P}_g = \{
\begin{pmatrix} \eta' - \omega' \omega^{-1} \eta & 0 \\ \eta & \omega \end{pmatrix}
\mid \Omega J \Omega^t = 2 \imath \pi J \}.
\end{gather*}

\subsection{Hyperelliptic locus}\label{ss:HEllLoc}
A genus $g$ hyperelliptic curve $\mathcal{C}$ is defined by
\begin{gather}\label{Cg}
 0 = f(x,y) = -y^2 +  \alpha_0 \prod_{j=0}^{2g+1} (x-e_j).
\end{gather}
Branch points $\{(e_j,0)\}$ are supposed to be all distinct if a curve is non-degenerate,
that is its genus equals $g$. For brevity, we denote branch points by their first coordinate $e_j$.
Without loss of generality we assign $ \alpha_0 =1$, since rescaling the $y$-coordinate does not affect 
the period matrix $\tau$. One of branch points can be sent to infinity, say $(e_0,0)$.

Every curve \eqref{Cg} corresponds to a binary form of degree $2g+2$:
$$\alpha(x,z) = \alpha_0 \prod_{j=0}^{2g+1} (x-e_j z).$$
with roots  $e_j$, $j=0$, \ldots, $2g+1$.

Homology basis is defined after H.\,Baker \cite[p.\,303]{bak898}.
One can imagine a continuous path through all branch points, which starts at $e_0$ and ends at $e_{2g+1}$.
The branch points are numerated along the path. The point $e_0$ serves as the base-point. 

Cuts are made between points $e_{2k-1}$ and $e_{2k}$ with $k$ from $1$ to $g$. 
One more cut $[e_{2g+1}, e_0]$ goes through infinity.
Canonical homology cycles are defined as follows.
Each $\mathfrak{a}_k$-cycle encircles the cut $(e_{2k-1},e_{2k})$, $k=1$, \ldots $g$,
and each $\mathfrak{b}_k$-cycle emerges from the cut $(e_{2k-1},e_{2k})$ and enters the cut $(e_{2g+1},e_{0})$,
see fig.~\ref{cycles}. 
\begin{figure}[h]
\includegraphics[width=0.6\textwidth]{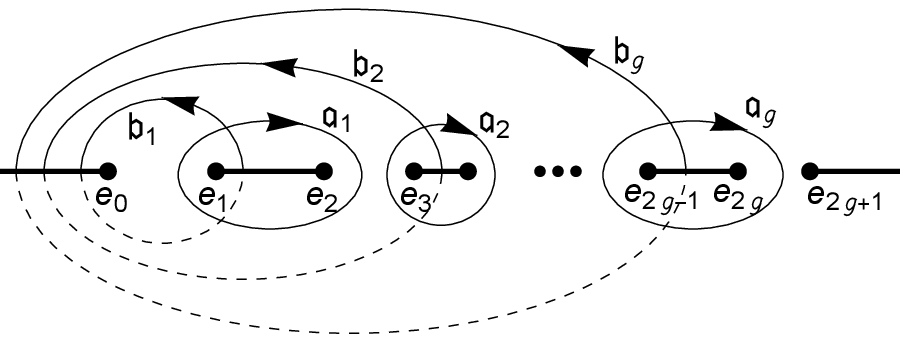}
\caption{} \label{cycles}
\end{figure}

The standard cohomology basis is employed, see for example \cite[p.\,306]{bak898}, 
and denoted  $\rmd u = (\rmd u_1$, $\rmd u_2$, $\dots$, $\rmd u_{g})$.
These differentials give rise to not normalized period matrices.

\subsection{Theta and sigma functions}\label{ss:ThetaChar}
Each curve of the family $\mathcal{C}$ has a Jacobian variety $\Jac(\mathcal{C})=\Complex^g/ \Lambda$,
which is a quotient space of $\Complex^g$ by the lattice $\Lambda$ 
of periods formed by columns of the matrix $(\omega'{}^t, \omega^t)$.
We denote not normalized coordinates on $\Jac(\mathcal{C})$ by $u=(u_1,$ $u_2$, \ldots, $u_{g})$. 
The  normalized periods are $(\tau, 1_g)$, where $\tau = \omega' \omega^{-1}$,
they define a normalized lattice and the corresponding Jacobian variety $\mathfrak{J}_g$. The 
normalized coordinates  are $v = u \omega^{-1}$, where
$v=(v_1$, $v_2$, \ldots, $v_g)$. 

The theta function is an entire function on $\Complex^g$ and normalized lattice:
\begin{gather}\label{ThetaDef}
 \theta(v;\tau) = \sum_{n\in \Integer^g} \exp \big(\imath \pi n^t \tau n + 2\imath \pi  v n\big).
\end{gather}

Abel's map $\mathcal{A}$ maps the curve to its Jacobian
\begin{gather}\label{AbelM}
 \Jac(\mathcal{C}) \ni \mathcal{A}(P) = \int_{e_0}^P \rmd v,\qquad P\in \mathcal{C}.
\end{gather}
Recall that we have chosen $e_0$ to be the base-point.
Abel's map of a positive divisor $\mathcal{D} = \sum_{i=1}^n P_i$ 
on $\mathcal{C}$ is defined by
\begin{gather}\label{AbelMD}
 \mathcal{A}(\mathcal{D}) = \sum_{i=1}^n 
 \int_{e_0}^{P_i} \rmd v.
\end{gather}

Each branch point $e$ of a hyperelliptic curve \eqref{Cg} is identified with a half-period,
see \cite[\S\,202]{bak897}
\begin{gather}\label{AbelMBranchP}
 \mathcal{A}(e) = \int_{e_0}^{e} \rmd v = \varepsilon/2 + \tau \varepsilon'/2,
 \qquad \begin{bmatrix} \varepsilon' \\ 
\varepsilon \end{bmatrix} = [\varepsilon],
\end{gather}
where components of $\varepsilon$ and $\varepsilon'$ are $0$ or $1$.
The integer $2{\times g}$-matrix $[\varepsilon]$ is the characteristic of the branch point $e$.
Characteristics are added $(\modR 2)$.
The theta function with characteristic $[\varepsilon]$ is defined by the formula
\begin{multline}\label{ThetaDef}
 \theta[\varepsilon](v;\tau) = \exp\big(\imath \pi (\varepsilon'/2) \tau \cdot (\varepsilon'/2)
 + 2 \imath \pi  (v+\varepsilon/2)\cdot(\varepsilon'/2)\big) \times \\ \times \theta(v+\varepsilon/2 + \tau \varepsilon'/2;\tau).
\end{multline}
A characteristic $[\varepsilon]$ is odd whenever $|\varepsilon|\equiv  \varepsilon \cdot \varepsilon'{}  = 1$ $(\modR 2)$, 
and even whenever $|\varepsilon|= 0$ $(\modR 2 )$. The theta function with characteristic
has the same parity as its characteristic.

\subsection{Characteristics in hyperelliptic case}\label{ss:CharHyper}
The method of constructing characteristics in hyperelliptic case is adopted from \cite[p.\,1012]{ER2008}.
It is based on the definition \eqref{AbelMBranchP} of half-period characteristics.
Let $[\varepsilon_k]$ be the characteristic of a branch point $e_k$.
Evidently, $[\varepsilon_{0}]=0$. Guided by the picture of canonical homology cycles, one can find
\begin{align*}
 &\mathcal{A}(e_{2g+1}) = \mathcal{A}(e_{0}) + \sum_{k=1}^g \int_{e_{2k-1}}^{e_{2k}} \rmd v &
 &[\varepsilon_{2g+1}] = \begin{bmatrix} 0 & 0 & \dots & 0 & 0 \\ 1& 1& \dots & 1 & 1\end{bmatrix}, & \\
 &\mathcal{A}(e_{2g}) = \mathcal{A}(e_{2g+1}) - \int_{e_{2g}}^{e_{2g+1}} \rmd v &
 &[\varepsilon_{2g}] = \begin{bmatrix} 0 & 0 & \dots & 0 & 1 \\ 1& 1& \dots & 1 & 1\end{bmatrix}, & \\
 &\mathcal{A}(e_{2g-1}) = \mathcal{A}(e_{2g}) - \int^{e_{2g}}_{e_{2g-1}} \rmd v&
 &[\varepsilon_{2g+1}] = \begin{bmatrix} 0 & 0 & \dots & 0 & 1 \\ 1& 1& \dots & 1 & 0 \end{bmatrix}, &
\intertext{for $k$ from $g-1$ to $2$}
 &\mathcal{A}(e_{2k}) = \mathcal{A}(e_{2k+1}) - \int_{e_{2k}}^{e_{2k+1}} \rmd v &
 &[\varepsilon_{2k}] = \text{\parbox[b]{4.85cm}{
$ \begin{matrix} 
 \multicolumn{4}{c}{\ \  \overbrace{\rule{2cm}{0pt}}^{k-1}} \end{matrix}$ \\
$ \begin{bmatrix} 
 0 & 0 & \dots & 0 & 1 & 0 & \dots & 0 \\ 
 1& 1& \dots & 1 & 1 & 0 & \dots & 0 \end{bmatrix}$ }}, & \\
 &\mathcal{A}(e_{2k-1}) = \mathcal{A}(e_{2k}) - \int^{e_{2k}}_{e_{2k-1}} \rmd v&
 &[\varepsilon_{2k-1}] = \text{\parbox[b]{4.85cm}{
$ \begin{matrix} 
 \multicolumn{4}{c}{\ \  \overbrace{\rule{2cm}{0pt}}^{k-1}} \end{matrix}$ \\
$ \begin{bmatrix} 
 0 & 0 & \dots & 0 & 1 & 0 & \dots & 0 \\ 
 1& 1& \dots & 1 & 0 & 0 & \dots & 0 \end{bmatrix}$ }}, &
\intertext{and finally}
 &\mathcal{A}(e_{2}) = \mathcal{A}(e_{3}) - \int_{e_{2}}^{e_{3}} \rmd v &
 &[\varepsilon_{2}] = \begin{bmatrix} 1 & 0 & \dots & 0  \\ 1& 0& \dots & 0 \end{bmatrix}, & \\
 &\mathcal{A}(e_{1}) = \mathcal{A}(e_{2}) - \int^{e_{2}}_{e_{1}} \rmd v &
 &[\varepsilon_{1}] = \begin{bmatrix} 1 & 0 & \dots & 0  \\ 0& 0& \dots & 0 \end{bmatrix}. &
 \end{align*}
This set of characteristics is azygetic in three and serves as a fundamental set, see 
\cite[\S\,202]{bak897}.

The characteristic $[K]$ equals 
the sum of all odd characteristics of branch points. There exist $g$ such characteristics,
see \cite[\S\,200--202]{bak897}.
Actually,
\begin{gather*}
 [K] = \sum_{k=1}^g [\varepsilon_{2k}].
\end{gather*}

\subsection{Characteristics and partitions}\label{ss:Part}
Let $\mathcal{I}\cup \J$  be a partition of the set of indices of all branch points $\{0,1,2,\dots, 2g+1\}$.
Denote by $[\varepsilon(\I)] = \sum_{i\in\I} [\varepsilon_i]$ the characteristic of
\begin{gather*}
 \mathcal{A} (\I) = \sum_{i\in\I} \mathcal{A}(e_i) = \frac{1}{2} \varepsilon_\I  + \frac{1}{2} \tau \varepsilon'_\I.
\end{gather*}

Characteristics of $2g+2$ branch points of \eqref{Cg} serve as a basis for constructing
all $2^{2g}$ half-period characteristics.
According to  \cite[\S\,202]{bak897} 
all half-period characteristics are represented by 
partitions of $2g+2$ indices of the form $\I_\mFr \cup \J_\mFr$ with $\I_\mFr = \{i_1,\,\dots,\, i_{g+1-2\mFr}\}$
and $\J_\mFr = \{j_1,\,\dots,\, j_{g+1+2\mFr}\}$, where $\mFr$ runs from $0$ to $[(g+1)/2]$, and $[\cdot]$ means the integer part. 
Number $\mFr$ is called \emph{multiplicity}, since 
$\theta(v+\mathcal{A} (\I_\mFr)+K)$ vanishes to order $\mFr$ at $v=0$,
according to the Riemann vanishing theorem.

Let $[\I_\mFr] \equiv [\delta(\I_\mFr)]  = [\varepsilon(\I_\mFr)] + [K]$, namely 
\begin{gather*}
 \sum_{i\in\I_\mFr} \mathcal{A}(e_i) + K = \frac{1}{2}\delta_{\I_\mFr} + \frac{1}{2}\tau \delta'_{\I_\mFr}.
\end{gather*}
The characteristic $[\I_\mFr]$  corresponds to the partition $\I_\mFr \cup \J_\mFr$. 
Note that $[\J_\mFr]$ represents the same characteristic as $[\I_\mFr]$.
Characteristics $[\I_\mFr]$ of even multiplicity $\mFr$ are even, and of odd $\mFr$ are odd.
Characteristics of multiplicity $0$ are called \emph{non-singular even characteristics},
there are $\binom{2g+1}{g}$ such characteristics.
There exist $\binom{2g+2}{g-1}$ characteristics of multiplicity $1$, which are called \emph{non-singular odd}.
All other characteristics are called \emph{singular}.
The number of characteristics of multiplicity $\mFr > 1$ is $\binom{2g+2}{g+1-2\mFr}$.

Characteristic $[K]$ corresponds to the partition $\{0\}\cup\{1,2,\dots,2g+1\}$ if $g$ is even,
or the partition $\emptyset \cup\{0, 1,2,\dots,2g+1\}$ if $g$ is odd. Such a partition is always unique, and
$\theta[K](v)$ vanishes to the maximal order $[(g+1)/2]$ at $v=0$.
In what follows, representation of characteristics in terms of partitions is preferable, 
because this makes clear which order of vanishing a theta function has at $v=0$.

Let a collection of branch points $\{e_i\mid i\in \mathcal{I}\}$ 
correspond to a partition $\I \cup \J$.
Then  $\Delta(\I)$ denotes the Vandermonde determinant in branch points from $\I$:
\begin{gather*}
 \Delta(\I) = \prod_{\substack{i > l \\ i,l \in \I}} (e_i-e_l).
\end{gather*}
Also  $s_n(\I)$ denotes the elementary symmetric polynomial of degree $n$ in  $\{e_i\mid i\in\I\}$,
Recall that  $s_n$ are defined by
\begin{gather*}
\sum_{n \geqslant 0} t^n s_n = \prod_{i\geqslant 0} (1+e_i t).
\end{gather*}

We denote all indices $\{i \mid i=0,\dots, 2g+1\}$ by $\I_{\text{All}}$, and
 the Vandermonde determinant in all branch points by $\Delta$.

\begin{rem}\label{R:NormMOrd}
 Let branch points in all factors $(e_i-e_l)$ be ordered in such a way that $i>l$, 
 we call this  \emph{normal ordering}. The operator $\Ord$ is used to introduce 
 the normal ordering in such factors.
 This allows to fix the multiplier $\epsilon$, $\epsilon^8=1$, which arises in many relations. 
 Such ordering was suggested by Baker \cite[p.\,346]{bak898}; in the case of all real branch points
 this guarantees that any Vandermonde determinant is positive.
\end{rem}

\subsection{G\"{o}pel groups and systems}
We briefly recall the theory of groups of characteristics developed by Frobenuis \cite{Frob1,Frob2},
see also \cite[Ch.\,XVII]{bak897}. 

Let $P=\big[\begin{smallmatrix} p' \\ p \end{smallmatrix} \big]$ and 
$Q=\big[\begin{smallmatrix} q' \\ q \end{smallmatrix} \big]$, then
$$ |P,Q| \overset{\text{def}}{=} pq'-p'q,\qquad  
|P,Q,R| \overset{\text{def}}{=} |P,Q| + |P,R| + |Q,R|. $$
$|P,Q|$ serves as an alternating bilinear form in the group of characteristics.
Two characteristics $P$ and $Q$ are called \emph{syzygetic} or 
\emph{azygetic} according as $|P,Q| = 0$ or $=1$ $(\modR 2)$.
Three characteristics $P$, $Q$, and $R$ are called \emph{syzygetic} or 
\emph{azygetic} according as $|P,Q,R| = 0$ or $=1$ $(\modR 2)$.

A group of pairwise zysygetic characteristics  is called a G\"{o}pel group.
Every G\"{o}pel group is generated by a finite number $r\leqslant g$ of characteristics, 
where $g$ is the number of columns in a 
characteristic also called the genus. 
A G\"{o}pel group generated by $r$ characteristics has rank $r$, and
contains $2^r$ elements. Employing the notation of  \cite{bak897}, we denote by $(P)$ a 
G\"{o}pel group of rank $r$ generated from $P_1$, $P_2$, \ldots, $P_r$.

By adding a characteristic $A$ to a G\"{o}pel group $(P)$ a G\"{o}pel system $(AP)$ is obtained.
G\"{o}pel systems derived from the same G\"{o}pel group have no characteristic in common,
they serve as  cosets of the G\"{o}pel  group.
Every G\"{o}pel group of rank $r$ split all $2^{2g}$ characteristics 
into $2^{2g-r}$  G\"{o}pel systems. 
Every three characteristics of a G\"{o}pel system are syzygetic, \cite[Art.\,297]{bak897}, 
similarly to a G\"{o}pel group.
There exist three types of G\"{o}pel systems: wholly even, wholly odd, and 
half odd-half even.  Let $s=g-r$.
There are $2^{s-1}(2^s+1)$ wholly even systems,  $2^{s-1}(2^s-1)$ wholly odd, and the remaining
$2^{2s}(2^r-1)$ systems are half odd-half even, see \cite[Art.\,299--300]{bak897} for more details.
If $r=g$, there exists one wholly even G\"{o}pel system and no odd G\"{o}pel system. 
If $r=g-1$, there are three  wholly even G\"{o}pel system and one odd G\"{o}pel system.

\subsection{Notation of theta constants}
In what follows we work with theta constants 
\begin{align*}
 &\theta[\I_0] = \theta[\I_0](0;\tau),&
 \intertext{and theta derivatives}
 &\partial_{v_i} \theta[\I_1] = \frac{\partial}{\partial v_i}  \theta[\I_1](v;\tau)\big|_{v=0},&\\
 &\partial^2_{v_i,v_j} \theta[\I_2] 
 = \frac{\partial^2}{\partial v_i\partial v_j}  \theta[\I_2](v;\tau)\big|_{v=0}	,&\\
&\text{etc.}&
\end{align*}
These are functions of $\tau$, which is usually omitted.
Characteristics are given in terms of partitions, for example $[\I_0]$, $[\I_1]$, $[\I_2]$.

We  denote by $\partial_v \theta [\I_1]$ a vector of first order theta derivatives $\partial_{v_i} \theta [\I_1]$,
$i=1$, \ldots, $g$, and by $\partial_v^\mFr \theta[\I_\mFr]$ a tensor of order $\mFr$ 
of  $\mFr$-th order theta derivatives with respect to all combinations constructed from $g$ components of $v$.
In particular, $\partial_v^2 \theta [\I_2]$ denotes the Hessian matrix of $\theta [\I_2]$.

Recall that theta function satisfies the system of heat equations
\begin{gather}\label{HEqTheta}
\partial_v^2 \theta [\I_2] = 4 \imath \pi \partial_\tau \theta [\I_2],
\end{gather}
here $\partial_\tau$ denotes a matrix operator with entries $\partial_{\tau_{i,j}}$.

\subsection{Thomae Formulas}\label{ThomaeF}
We recall the first and the second Thomae formulas, see \cite{ER2008,Tho870}, 
which give expressions for theta constants and first order theta derivatives. 
We also need expressions for second order theta derivatives, which are
given by the general Thomae formula, see \cite{bGTF}. All Thomae formulas below
are accommodated to the case of all finite branch points.

\newtheorem*{FTteo}{First Thomae theorem}
\begin{FTteo}
 Let $\I_0\cup \J_0$ with $\I_0=\{0$, $i_1$, \ldots, $i_{g}\}$ 
 and $\J_0 = \{j_1$, \ldots, $j_{g+1}\}$  be a partition
 of the set $\{0$, $1$, \ldots, $2g+1\}$ of indices of branch points, and  $[\I_0]$ denote 
 the non-singular even characteristic corresponding to $\mathcal{A} (\I_0) + K$.  Then
 \begin{gather}\label{thomae1}
 \theta[\I_0] = \epsilon \bigg(\frac{\det \omega}{\pi^g}\bigg)^{1/2} \Delta(\I_0)^{1/4} \Delta(\J_0)^{1/4}.
\end{gather}
where $\epsilon$ satisfies $\epsilon^8=1$,
and $\Delta(\I_0)$, $\Delta(\J_0)$ denote Vandermonde determinants 
built from $\{e_i\mid i\in \I_0\}$ and $\{e_j\mid j\in \J_0\}$.
\end{FTteo}

\newtheorem*{STteo}{Second Thomae theorem}
\begin{STteo}
 Let $\I_1\cup \J_1$ with $\I_1=\{i_1$, \ldots, $i_{g-1}\}$ and $\J_1 = \{j_1,\,\dots,\,j_{g+3}\}$ 
 be a partition of the set $\{0,1,2,\dots,2g+1\}$ of indices of branch points, and
 $[\I_1]$ denote the non-singular odd characteristic  
 corresponding to $\mathcal{A} (\I_1) + K$. Then for any $n \in \{1, \dots, g\}$
 \begin{equation}\label{thomae2}
 \partial_{v_n} \theta[\I_1] 
 = \epsilon \bigg(\frac{\det \omega}{\pi^g}\bigg)^{1/2} \Delta(\I_1)^{1/4} \Delta(\J_1)^{1/4} 
 \sum_{j=1}^g  (-1)^{j-1} s_{j-1}(\I_1) \omega_{n,j}.
\end{equation}
where  $\epsilon$ satisfies $\epsilon^8=1$, and $\Delta(\I_1)$, $\Delta(\J_1)$ denote
Vandermonde determinants built from $\{e_i\mid i\in \I_1\}$ and $\{e_j\mid j\in \J_1\}$,
and $s_{j}(\I)$ denotes the elementary symmetric polynomial of degree $j$ in $\{e_i\mid i\in \I\}$.
\end{STteo}

Taking into account the relation between normalized variables $v$ and not normalized variables $u$,
namely $u = v \omega$, $\partial_v = \omega \partial_u$,
the formula \eqref{thomae2} can be rewritten as follows
\begin{equation}\label{thomae2MFu}
 \begin{pmatrix} \partial_{u_1} \\ \partial_{u_2}  \\ \vdots \\ \partial_{u_{g}}  \end{pmatrix}  
 \theta[\I_1]
 =  \epsilon \bigg(\frac{\det \omega}{\pi^g}\bigg)^{1/2} \Delta(\I_1)^{1/4} \Delta(\J_1)^{1/4}
 \begin{pmatrix} s_0(\I_1) \\ - s_1(\I_1) \\ \vdots \\ (-1)^{g-1} s_{g-1}(\I_1) 
\end{pmatrix}.
\end{equation}

In the case of second order theta derivatives
\newtheorem*{GTteo}{Third Thomae theorem}
\begin{GTteo}
Let $\I_2\cup \J_2$ with $\I_1=\{i_1$, \ldots, $i_{g-3}\}$ and $\J_1 = \{j_1,\,\dots,\,j_{g+5}\}$ 
 be a partition of the set $\{0,1,2,\dots,2g+1\}$ of indices of branch points, and
 $[\I_2]$ denote the a characteristic of multiplicity $\mFr=2$ 
 corresponding to $\mathcal{A} (\I_2) + K$. Then for any $m$, $n \in \{1, \dots, g\}$
 \begin{gather}\label{thomaeM2}
 \partial^2_{v_m, v_n} \theta[\I_2] 
 = \epsilon \bigg(\frac{\det \omega}{\pi^g}\bigg)^{1/2} \Delta(\I_2)^{1/4} \Delta(\J_2)^{1/4} 
 \sum_{i,j=1}^g S_{i,j} \omega_{m,i} \omega_{n,j},\\
 S_{i,j} = (-1)^{i+j} \big(2 s_{i-2}(\I_2) s_{j-2}(\I_2) - s_{i-1}(\I_2) s_{j-3}(\I_2) - s_{i-3}(\I_2) s_{j-1}(\I_2) \big),
 \label{Smatr}
\end{gather}
where  $\epsilon$ satisfies $\epsilon^8=1$, and $\Delta(\I_2)$, $\Delta(\J_2)$ denote
Vandermonde determinants built from $\{e_i\mid i\in \I_2\}$ and $\{e_j\mid j\in \J_2\}$,
and $s_{k}(\I)$ denotes the elementary symmetric polynomial of degree $k$ in $\{e_i\mid i\in \I\}$,
$s_{k}(\I)=0$ if $k<0$.
\end{GTteo}

With non-normalized variables $u$
 \begin{gather}\label{thomae3MFu}
 \partial_{u}^2  \theta[\I_2](\omega^{-1} u) \big|_{u=0} 
 = \epsilon \bigg(\frac{\det \omega}{\pi^g}\bigg)^{1/2} \Delta(\I_2)^{1/4} \Delta(\J_2)^{1/4} \hat{S}[\I_2],
\end{gather}
where $\hat{S}$ denotes the matrix with entries $S_{i,j}$.

\begin{rem}\label{R:NormThomaeF}
If the normal ordering $\Ord$ on factors $(e_i-e_l)$ is applied, as explained in Remark~\ref{R:NormMOrd}, 
then $\epsilon$ is the same for all characteristics of the type,
namely
 \begin{align}
& \theta[\I_0] {=}  \bigg(\frac{\det \omega}{\pi^g}\bigg)^{1/2} 
 \Ord \big(\Delta(\I_0) \Delta(\J_0) \big)^{1/4},  \tag{\ref{thomae1}'}\label{thomae1Ord} \\
& \partial_{v_n} \theta[\I_1]
 {=} (-1)^g \bigg(\frac{\det \omega}{\pi^g}\bigg)^{1/2} \Ord \big(\Delta(\I_1) \Delta(\J_1)\big)^{1/4} 
 \sum_{j=1}^g  (-1)^{j-1} s_{j-1}(\I_1) \omega_{n,j},  \tag{\ref{thomae2}'}\label{thomae2Ord} \\
& \partial^2_{v_m, v_n} \theta[\I_2] 
 {=} - \bigg(\frac{\det \omega}{\pi^g}\bigg)^{1/2} \Ord \big(\Delta(\I_2) \Delta(\J_2)\big)^{1/4} 
 \sum_{i,j=1}^g S_{i,j}   \omega_{m,i} \omega_{n,j}.
\tag{\ref{thomaeM2}'}\label{thomaeM2Ord} 
\end{align}
\end{rem}

\section{Transformation of theta derivatives}\label{s:ThetaTransf}
As mentioned in Introduction, we consider modular forms vanishing on the hyperelliptic locus.
We construct such forms from monomials in theta constants with even characteristics, 
including singular. In genera $3$ and $4$ the highest multiplicity of characteristics is two.
Thus, all singular even characteristics are of the type $[\I_2]$.

As we focus our attention on vanishing modular forms, we will investigate their lowest 
non-vanishing derivatives. Such derivatives are \emph{vector-valued hyperelliptic modular forms},
since they obey transformation laws under the action of
the modular group, as shown below in this section.

\begin{Def}
We define the extended map $\rho$ on the ring of polynionals in theta constants and
theta derivatives within the hyperelliptic locus of $\mathfrak{S}_g$. 
The map brings such polynionals to the ring of homogeneous functions in roots of a binary $2g+2$-form,
images are given by Thomae formulas \eqref{thomae1}, \eqref{thomae2}, \eqref{thomaeM2}.
\end{Def}
Since in the present paper we consider genera not higher than $4$, 
the formula \eqref{thomaeM2}, called here the Third Thomae formula, is sufficient.
The reader may find a more general result, called the General Thomae formula, in \cite{bGTF}.

The proofs presented below
are based on Thomae formulas, which make them simple and elegant. Such an  approach 
can be considered as an alternative to the proof given in \cite{Ig1964}.
Since we work in the hyperelliptic locus, this type of proofs is reasonable and sufficient.

\begin{theo}\label{CMFGSI0}
Let $(A)$ be a collection of $\dd$ non-singular even characteristics, which 
gives rise to a monomial of non-vanishing theta constants of degree~$\dd$
\begin{gather}\label{GSTheta}
\psi_{\dd/2}(\tau; (A)) =  \prod_{[\I_0] \in (A)} \theta[\I_0].
\end{gather}
Then $\psi_{\dd/2}$ is a monomial is a modular form of weight $\dd/2$: 
\begin{gather}\label{I0PsiModTrans}
\forall \gamma \in \Gamma_g \quad \psi_{\dd/2}(\gamma \langle \tau \rangle; \gamma^{-1}(A)) = 
\epsilon^{\dd} \det (c \tau+d)^{\dd/2} \psi_{\dd/2}(\tau;  (A)),
\end{gather}
where $\epsilon^8 = 1$.
\end{theo}

\begin{proof}
Applying the First Thomae theorem to $f_{\dd/2}(\tau; (A))$, we obtain
\begin{gather}\label{GSThomae}
\psi_{\dd/2}(\tau; (A)) = \Big(\frac{\det \omega}{\pi^g}\Big)^{\dd/2} 
\prod_{[\I_0] \in (A)}  \big(\Delta[\I_0] \Delta[\J_0]\big)^{1/4},
\end{gather}
where $\omega$ is a not normalized period matrix. 
Recall, how $\Gamma_g$ acts on a characteristic,
see \cite[p.\,226]{Ig1964}:
\begin{align*}
 & [\gamma^{-1}\varepsilon] = 
 \begin{bmatrix} d \varepsilon' - c \varepsilon + \diag (c^t d) \\
- b \varepsilon'  + a \varepsilon + \diag (a^t b) \end{bmatrix},
\end{align*}
where $\diag (\cdot)$ means taking the main diagonal of a square matrix in the argument.
Now, we apply 
a modular transformation $\gamma$ to $\tau$ and $\omega$, 
according to  \eqref{GammaTau},
and the complementary action
to characteristics $\gamma^{-1} (A)$, and find 
\begin{gather*}
\psi_{\dd/2} \big(\gamma \langle \tau \rangle; \gamma^{-1} (A) \big) = 
\epsilon^{\dd} \big(\!\det (c\tau+d) \big)^{\dd/2} \Big(\frac{\det \omega}{\pi^g}\Big)^{\dd/2} 
\prod_{[\I_0] \in (A)}  \big(\Delta[\I_0] \Delta[\J_0]\big)^{1/4}.
\end{gather*}
Simplifying the right hand side by means of \eqref{GSThomae}, we   
come to \eqref{I0PsiModTrans}.
\end{proof}

\begin{rem}
The formula \eqref{GSThomae} defines the map $\rho$  from monomials in theta constants into
algebraic functions of the roots $\{e_i \mid i=0,\dots,2g+1\}$ of a binary form $f(x,0)$ of degree $2g+2$,
cf. \eqref{Cg}. Note, that Thomae formulas, as well as \eqref{GSThomae},
determine unambiguously which forth root of $\Delta[\I_0] \Delta[\J_0]$ is taken: the one
which gives equality with theta functions on the left hand side. We assume, that 
the root is always chosen correctly.
\end{rem}

\begin{lemma}\label{CMFI2}
Let $[\I_2]$ be a singular even characteristic of multiplicity $2$, which vanishes
 to the first order on the hyperelliptic locus of $\mathfrak{S}_g$.
Then entries of $\partial_{\tau} \theta [\I_2 ]$ obey the transformation law
\begin{gather}\label{I2ModTrans}
\forall \gamma\in \Gamma_g \quad
\partial_{\tau} \theta [\gamma^{-1}  \I_2 ] (\gamma \langle \tau \rangle) = 
\epsilon \big(\!\det (c\tau+d) \big)^{1/2} (c\tau+d)
\big( \partial_{\tau} \theta [\I_2 ] (\tau) \big) (c\tau+d)^t,
\end{gather}
where $\epsilon^8=1$.
\end{lemma}
\begin{proof}
From the Third Thomae theorem, see subsection~\ref{ThomaeF}, we know, that
\begin{gather}\label{GSThomaeSng}
\partial^2_{v} \theta [\I_2] = - \Big(\frac{\det \omega}{\pi^g}\Big)^{1/2} 
  \big(\Delta[\I_2] \Delta[\J_2]\big)^{1/4} \omega \hat{S}(\I_2) \omega^t,
\end{gather}
in the matrix form. The matrix $\hat{S}(\I_2)$ has entries $S_{i,j}$ defined by \eqref{Smatr},
which depend only on $\{e_i | i \in \I_2\}$. Taking into account \eqref{GammaTau}, 
we find:
\begin{gather*}
\begin{split}
\partial^2_{v} \theta [\gamma^{-1}  \I_2 ] (\gamma \langle \tau \rangle) &= 
\big(\!\det (c \tau + d) \big)^{1/2} \Big(\frac{\det \omega}{\pi^g}\Big)^{1/2}   \times \\
  &\qquad \times 
  \big(\Delta[\I_2 ] \Delta[\J_2 ]\big)^{1/4} 
  (c \tau + d) \big(\omega \hat{S}( \I_2 ) \omega^t \big) (c \tau + d)^t.
\end{split}
\end{gather*}
Finally,
\begin{gather*}
\partial^2_{v} \theta [\gamma^{-1}  \I_2 ] (\gamma \langle \tau \rangle) = 
\big(\!\det (c \tau + d) \big)^{1/2} (c \tau + d)
\big( \partial^2_{v} \theta [ \I_2 ] (\tau) \big) (c \tau + d)^t.
\end{gather*}
Since theta functions satisfy $\partial^2_{v} \theta = 4\imath \pi \partial_\tau \theta$, 
cf.\,\eqref{HEqTheta}, we obtain \eqref{I2ModTrans}.
\end{proof}

\begin{theo}\label{CMFGSI2}
Let $(A)$ be a collection of $\dd-1$ non-singular even characteristics 
and one singular even characteristic $[\I_2]$ of multiplicity $2$. 
Then  $(A)$ gives rise to a theta monomial of degree $\dd$
\begin{gather}\label{GSThetaSng}
\psi_{\dd/2;1}(\tau; (A)) =  \theta[\I_2] \prod_{[\I_0] \in (A)} \theta[\I_0],
\end{gather}
which vanishes to the first order on the hyperelliptic locus of $\mathfrak{S}_g$;
entries of $\partial_{\tau} \psi_{\dd/2;1}$  obey  the transformation law
\begin{multline}\label{I2PsiModTrans}
\forall \gamma \in \Gamma_g \quad
\partial_{\tau} \psi_{\dd/2;1}(\gamma\langle \tau \rangle;\gamma^{-1}(A)) \\ = 
\epsilon^{\dd} \big(\!\det (c\tau+d) \big)^{1/2} (c \tau+d)
\big( \partial_{\tau}   \psi_{\dd/2;1}(\tau; (A)) \big) (c \tau+d)^t,
\end{multline} 
where  $\epsilon^8=1$.
\end{theo}
\begin{proof}
Differentiating \eqref{GSThetaSng}, we find
\begin{gather}\label{GSThetaSngDer}
\partial_\tau \psi_{\dd/2;1}(\tau; (A)) =  (\partial_\tau \theta[\I_2]) 
\prod_{[\I_0] \in (A)} \theta[\I_0].
\end{gather}
Here, $\partial_\tau \theta[\I_2]$ obeys Lemma~\ref{CMFI2},
and the remaining product  obeys Theorem~\ref{CMFGSI0}.
Combining the two transformations \eqref{I2ModTrans} and \eqref{I0PsiModTrans},
 we obtain \eqref{I2PsiModTrans}.
\end{proof}

\begin{lemma}\label{CMF2I2}
Let $[\I_2^{(1)}]$, $[\I_2^{(2)}]$ be  two singular even characteristic of multiplicity~$2$.
Then the product $\theta [\I_2^{(1)} ]\,  \theta [\I_2^{(2)} ]$ vanishes to the second order 
on the hyperelliptic locus of $\mathfrak{S}_g$, and all entries of 
$\partial^{\otimes 2}_{\tau} \big(\theta [\I_2^{(1)} ]\,  \theta [\I_2^{(2)} ]\big)$ 
obey the  transformation law
\begin{multline}\label{2I2ModTrans}
\forall \gamma \in \Gamma_g \quad \partial^{\otimes 2}_{\tau} 
\big(\theta [\gamma^{-1} \I_2^{(1)}](\gamma \langle \tau \rangle)  
\,\theta [\gamma^{-1} \I_2^{(2)}] (\gamma \langle \tau \rangle) \big) = 
\epsilon^2 \big(\!\det (c \tau+d) \big) \times \\ \times
  (c \tau+d)^{\otimes 2} \partial^{\otimes 2}_{\tau} \big(\theta [\I_2^{(1)}](\tau)
\,\theta [\I_2^{(2)}] (\tau) \big)
((c \tau+d)^t)^{\otimes 2},
\end{multline}
where  $(c \tau+d)^{\otimes 2}= (c \tau+d) \otimes (c \tau+d)$, and $\epsilon^8=1$.
\end{lemma}
\begin{proof}
Recall that  $\partial^{\otimes 2}_{\tau} = (\partial_\tau) \otimes (\partial_\tau)$
is an order $4$ tensor operator.
Evidently, $\theta [\I_2^{(1)} ]\,  \theta [\I_2^{(2)} ]$ vanishes to the second order, 
since each factor
$\theta[\I_2^{(1)}]$, $\theta[\I_2^{(2)}]$ vanishes to the first order, and
\begin{gather*}
\partial^{\otimes 2}_{\tau}\big(\theta[\I_2^{(1)}] \theta[\I_2^{(2)}] \big) = 
\big(\partial_\tau \theta[\I_2^{(1)}]\big) \otimes \big( \partial_\tau \theta[\I_2^{(2)}]\big) 
+ \big(\partial_\tau \theta[\I_2^{(2)}]\big) \otimes \big( \partial_\tau \theta[\I_2^{(1)}]\big).
\end{gather*}
Then we apply  the Third Thomae theorem and obtain
\begin{multline}\label{GSThomae2Sng}
\partial^{\otimes 2}_{\tau}\big(\theta[\I_2^{(1)}] \theta[\I_2^{(2)}] \big) =
\frac{1}{(4\imath \pi)^2} 
\Big(\frac{\det \omega}{\pi^g}\Big)
\big(\Delta[\I_2^{(1)}] \Delta[\J_2^{(1)}] \Delta[\I_2^{(2)}] \Delta[\J_2^{(2)}]\big)^{1/4} \times \\ 
\times \big(\omega \otimes \omega\big)
\Big( \hat{S}(\I_2^{(1)}) \otimes  \hat{S}(\I_2^{(2)}) 
+  \hat{S}(\I_2^{(2)}) \otimes \big(\hat{S}(\I_2^{(1)}) \Big) \big(\omega^t \otimes \omega^t\big),
\end{multline}
where multiplication on tensor products obeys the rule $(a\otimes b) (c \otimes d) = (a c) \otimes (bd)$.
Taking into account \eqref{GammaTau}, 
we find the transformation \eqref{2I2ModTrans} under the action of $\Gamma_g$.
\end{proof}

Lemma~\ref{CMF2I2} is easily extended to a product of an arbitrary number 
of theta constants with singular even characteristics.

\begin{theo}\label{CMFGS2I2}
Let $(A)$ be collection of $\dd-2$ non-singular even characteristics 
and two singular even characteristic $[\I_2^{(1)}]$, $[\I_2^{(2)}]$ of multiplicity $2$. 
Then  $(A)$ gives rise to a theta monomial of degree $\dd$
\begin{gather}\label{GSTheta2Sng}
\psi_{\dd/2;2}(\tau; (A)) 
= \theta[\I_2^{(1)}] \theta[\I_2^{(2)}] \prod_{[\I_0] \in (A)} \theta[\I_0],
\end{gather}
which vanishes to the second order on the hyperelliptic locus of $\mathfrak{S}_g$,
and entries of $\partial^{\otimes 2}_{\tau} \psi_{\dd/2;2}$
obeys the transformation law
\begin{multline}\label{2I2PsiModTrans}
\partial^{\otimes 2}_{\tau} \psi_{\dd/2;2}(\gamma \langle \tau \rangle; \gamma^{-1}(A)) = 
\epsilon^{\dd} \big(\!\det (c\tau+d) \big)^{\dd/2} \times \\
\times  (c\tau+d)^{\otimes 2}
\big(\partial^{\otimes 2}_{\tau} \psi_{\dd/2;2}(\tau; (A)) \big)
((c\tau+d)^t)^{\otimes 2},
\end{multline}
where $(c\tau+d)^{\otimes 2}= (c\tau+d) \otimes (c\tau+d)$, and $\epsilon^8=1$.
\end{theo}
\begin{proof}
Differentiating \eqref{GSTheta2Sng}, we find
\begin{gather}\label{GSTheta2SngDer}
\partial_\tau \psi_{\dd/2;2}(\tau; (A)) =  \partial^{\otimes 2}_{\tau}\big(\theta[\I_2^{(1)}] \theta[\I_2^{(2)}] \big)
\prod_{[\I_0] \in (A)} \theta[\I_0].
\end{gather}
Here $\partial^{\otimes 2}_{\tau}\big(\theta[\I_2^{(1)}] \theta[\I_2^{(2)}] \big)$ obeys Lemma~\ref{CMF2I2},
and the remaining product  obeys Theorem~\ref{CMFGSI0}.
Combining the two transformations \eqref{2I2ModTrans} and \eqref{I0PsiModTrans},
 we obtain \eqref{2I2PsiModTrans}.
\end{proof}

In the next two sections we analyze vanishing monomials in even theta constants and
obtain vector-valued modular forms in genera $3$ and~$4$.

\section{Extension of $\rho$ in genus 3}

In genus $3$, among $2^3=64$ characteristics we have $36$ even and $28$ odd.
All odd characteristics are non-singular, that is of the type $[\I_1]$. 
There exists a unique singular even characteristic of multiplicity $2$, namely $[\I_2]=[\emptyset]$. 
The remaining $35$ even characteristics are non-singular, that is of the type $[\I_0]$.
If characteristics of branch points from subsection~\ref{ss:CharHyper} are employed,
the singular even characteristic is $[\emptyset] = \big({}^{111}_{101}\big)$.
In what follows we describe characteristics by means of partitions, as explained in subsection~\ref{ss:Part}.
We denote by $\I_{All}$ the set of indices $\{i\}_{i=0}^7$ of all roots of a binary form of degree $8$.

Using G\"{o}pel systems, we construct monomials in even theta constants.
We are interested in monomials which contain $\theta[\emptyset]$, vanishing 
on the hyperelliptic locus of $\mathfrak{S}_3$. Then we compute the lowest non-vanishing derivatives
of such monomials, which are not higher than the first order in genus $3$.
With the help of the First \eqref{thomae1}, and the Third Thomae formula \eqref{thomaeM2}, 
we find images of the obtained derivatives under the map $\rho$. 

As a result, we find that derivatives of vanishing monomials of weight $4$, obtained from G\"{o}pel systems 
of the highest rank $3$, map to the discriminant of a binary form of degree $8$.
Derivatives of $\chi_{18}$ exhibit similar behaviour.

\begin{lemma}\label{CMFI2G3}
On the hyperelliptic locus of $\mathfrak{S}_3$ 
the theta constant $\theta[\emptyset]$ with the singular even characteristic vanishes 
to the first order. The matrix  
$\partial_\tau \theta[\emptyset]$  obeys the transformation law \eqref{I2ModTrans}, and
\begin{gather}\label{Chi5s2}
\partial_\tau \theta[\emptyset] 
 = - \frac{\epsilon}{4\imath \pi}\Big(\frac{\det \omega}{\pi^3}\Big)^{1/2}\Delta^{1/4}
 \omega \hat{S}[\emptyset] \omega^t,
\end{gather}
where $\epsilon^8=1$, $\omega$ is a not nomalized period matrix, $\Delta$ denotes 
the Vandermonde determinant in $\{e_i\}_{i=0}^7$, and
$$\hat{S}[\emptyset] =  \begin{pmatrix} 0&0&-1 \\ 0&2&0 \\ -1&0&0 \end{pmatrix}. $$
\end{lemma}
\begin{proof}
This Lemma is a particular case of Lemma~\ref{CMFI2} when genus is $g=3$.
\end{proof}

\begin{rem}
Here and in what follows, we assume, that 
the roots $\{e_i\}_{i=0}^7$ are numbered in accordance with the location of
canonical cycles as on  fig.~\ref{cycles}. This guarantees the absence of the eighth root $\epsilon$ of the unity
in $\rho$-images, cf.\,\eqref{Chi5s2}. Otherwise, the multiple $\epsilon$ is found from the value of the left-hand side.
\end{rem}

\subsection{All even characteristics}
Let $\chi_{18}$ be defined by
\begin{equation}\label{Chi18Gen}
\chi_{18} = \prod_{\text{all even }[\I]} \theta[\I].
\end{equation}

\begin{theo}
On the hyperelliptic locus of $\mathfrak{S}_3$, $\chi_{18}$ acquires the form
\begin{equation}\label{Chi18}
\chi_{18} = \theta[\emptyset] \prod_{a=1}^{35} \theta[\I_0^{(a)}],
\end{equation}
and vanishes to the first order. Moreover, $\partial_{\tau} \chi_{18}$
is a vector-valued modular form of weight $20$, which obeys the transformation law \eqref{I2PsiModTrans},
and
\begin{gather}\label{DChi18Dtau}
\partial_{\tau} \chi_{18} = - \frac{1}{4\imath \pi}  \Big(\frac{\det \omega}{\pi^3}\Big)^{18}
 \Delta^4 \omega \hat{S}(\emptyset) \omega^t.
\end{gather}
\end{theo}
\begin{proof}
The formula \eqref{DChi18Dtau} is obtained by straightforward computation, based on
Lemma~\ref{CMFI2G3} and the First Thomae theorem. Evidently,
$\partial_{\tau} \chi_{18}$ satisfies \eqref{I2PsiModTrans}, and so is a vector-valued modular form.
The $\rho$-image is proportional to the forth power of the discriminant $\Delta$ 
on roots $\{e_i\}_{i=0}^7$ of a binary form of degree $8$.
\end{proof}

\subsection{G\"{o}pel systems of $8$ characteristics}
In genus $3$, there exist  $135$ G\"{o}pel groups of rank $3$, 
each with one wholly even G\"{o}pel system. 
Thus, $135$ collections of eight even characteristics are obtained. 
Among them $105$ G\"{o}pel systems consist of non-singular even characteristics, 
we denote this type by $8[\I_0]$; and 
$30$ contain the singular even characteristic $[\emptyset]$, 
we denote the latter type by $1[\I_2]+7[\I_0]$.

\begin{theo}\label{T:GSr3StructG3}
Let $(A P)$ be a wholly even G\"{o}pel system of $8$ characteristics in genus $3$
with the singular characteristic $[\emptyset]$. 
Then $(A P)$ is formed by characteristics
\begin{gather}\label{GSr3ch2I2p6I0G3}
\begin{split}
[\emptyset],  \qquad  &[\{i_1,i_2,i_3,i_4\}],\quad
[\{i_1,i_2,j_1,j_2\}],\quad [\{i_1,i_2,j_3,j_4\}],\\
[\{i_1,i_3,j_1,j_3\}],\quad &[\{i_1,i_3,j_2,j_4\}], \quad
[\{i_1,i_4,j_1,j_4\}],\quad  [\{i_1,i_4,j_2,j_3\}].
\end{split}
\end{gather}
where $i_1$, $i_2$, $i_3$, $i_4$, $j_1$, $j_2$, $j_3$, $j_4$ are different numbers chosen from 
$8$ indices $\I_{All}$.
\end{theo}
\begin{proof} is based on straightforward computations.

There are exist $30$  wholly even G\"{o}pel systems of $8$ characteristics
of the form \eqref{GSr3ch2I2p6I0G3}.
Indeed, four indices $\{i_1,i_2,i_3,i_4\}$ can be chosen from $\I_{\text{All}}$ in $70$ ways.
We split each four into a pair of two, and obtain $210$ combinations $\{i_1,i_2 \mid i_3,i_4\}$. 
These $210$ combinations split into $30$ collections of seven, \and each seven
combine into a G\"{o}pel system of the form \eqref{GSr3ch2I2p6I0G3}. 
\end{proof}

\begin{theo}\label{T:Chi4G3}
Every wholly even G\"{o}pel system $(AP)$ of $8$ characteristics 
which include the singular characteristic $[\emptyset]$, 
as defined in Theorem~\ref{T:GSr3StructG3},
gives rise to a monomial in theta constants 
\begin{gather}\label{Chi4G3}
\chi_4 (AP) = \theta[\emptyset] \prod_{[\I_0] \in (AP)} \theta[\I_0]. 
\end{gather}
The monomials $\chi_4$ vanish to the first order on the hyperelliptic locus of $\mathfrak{S}_3$, 
and $\partial_\tau \chi_4$ is a vector-valued modular form of weight $6$. Moreover,  $\rho$-images
of $\partial_\tau \chi_4 (AP)$ do not depend on the choice of a G\"{o}pel system $(AP)$:
\begin{gather}\label{DChi4}
\partial_\tau \chi_4  =  - \frac{1}{4\imath \pi} 
\Big(\frac{\det \omega}{\pi^3}\Big)^{4} \Delta \, \omega \hat{S}(\emptyset) \omega^t.
\end{gather}
\end{theo}
\begin{proof}
Evidently, the lowest non-vanishing derivative of $\chi_4(AP)$ is
\begin{gather*}
\partial_\tau \chi_4 (AP) = (\partial_\tau \theta[\emptyset]) \prod_{[\I_0] \in (AP)} \theta[\I_0].
\end{gather*}
Applying Lemma~\ref{CMFI2G3} and the First  Thomae theorem, we find \eqref{DChi4}.
So, $\partial_\tau \chi_4$ obeys by the transformation law \eqref{I2PsiModTrans}, and 
the $\rho$-image is proportional to the discriminant $\Delta$ 
on roots $\{e_i\}_{i=0}^7$ of a binary form of degree $8$.
Since the matrix $\hat{S}(\emptyset)$ is constant,  $\rho$-images of all $\partial_\tau \chi_4 (AP)$
coincide.
\end{proof}

\subsection{G\"{o}pel systems of $4$ characteristics}
In genus $3$, there are 
$315$ G\"{o}pel groups of rank $2$, each possesses three wholly even G\"{o}pel systems.
Among $315$ collections of three even G\"{o}pel systems, $210$ have
the type $3(4[\I_0])$, that is all three G\"{o}pel systems contain only non-singular even characteristics;
and the remaining $105$ have the type  $(1[\I_2]+3[\I_0])+2(4[\I_0])$, that is each three
contain one G\"{o}pel system of the type $[\I_2]+3[\I_0]$, with the singular even characteristic, 
and two G\"{o}pel systems of the type $4[\I_0]$, with non-singular even characteristics. 
The latter are of our particular interest.

\begin{theo}\label{T:GSr2StructG3}
Let $(A_k P)$, $k=1$, $2$, $3$, be three wholly even G\"{o}pel systems 
derived from the same G\"{o}pel group $(P)$ of rank $2$ in genus $3$,
such that one G\"{o}pel system, say $(A_1 P)$, contains the singular characteristic $[\emptyset]$, 
and the other two are composed of non-singular even characteristics. 
Then $(A_1 P)$ is formed by characteristics
\begin{gather}\label{GSr2ch2I2p2I0G3}
[\emptyset],  \qquad  [\{i_1,j_1,i_2,j_2\}],\quad
[\{i_1,j_1,i_3,j_3\}],\quad [\{i_1,j_1,i_4,j_4\}],
\end{gather}
where $\{i_1,j_1\} \cup\{i_2,j_2\}\cup\{i_3,j_3\}\cup \{i_4,j_4\}$ is a partition of 
$8$ indices $\I_{All}$, and 
the remaining two G\"{o}pel systems consist of characteristics
\begin{gather}\label{GSr2ch4I0G3}
\begin{split}
&(A_2 P):\ \  [\{i_1,i_2,i_3,i_4\}],\ [\{i_1,i_2,j_3,j_4\}],\ 
[\{i_1,j_2,i_3,j_4\}],\ [\{i_1,j_2,j_3,i_4\}], \\
&(A_3 P):\ \  [\{i_1,i_2,i_3,j_4\}],\  [\{i_1,i_2,j_3,i_4\}],\ 
[\{i_1,j_2,i_3,i_4\}],\  [\{j_1,i_2,i_3,i_4\}].
\end{split}
\end{gather}
\end{theo}
\begin{proof} is based on straightforward computations.

Three wholly even G\"{o}pel systems of $4$ characteristics
such that $(A_1 P)$ has the form \eqref{GSr2ch2I2p2I0G3} and
$(A_2 P)$, $(A_3 P)$ have the form \eqref{GSr2ch4I0G3} can be chosen
in $8!/(2!2!2!2!4!)\,{=}\,105$ ways from 8 indices $\I_{\text{All}}$. This is the number of 
collections of three G\"{o}pel systems  of the type $(1[\I_2]+3[\I_0])+2(4[\I_0])$ in genus $3$.
 \end{proof}
 
 Let $(AP)$ and $(BP)$ be two wholly even G\"{o}pel systems 
derived from the same G\"{o}pel group $(P)$ of rank $2$, such that 
$(AP)$ contains the singular characteristic $[\emptyset]$. 
As stated in Theorem~\ref{T:GSr2StructG3}, such a pair of G\"{o}pel systems 
is represented by $(A_1P)$, $(A_2 P)$ or $(A_1P)$, $(A_3 P)$.
Then we introduce a function
\begin{gather}\label{h0G3}
h_0(A,B,(P)) = \frac{ \theta[\emptyset]  \prod_{[\I_0] \in (A P)} \theta[\I_0]}
{\prod_{[\I_0] \in (B P)} \theta[\I_0]}.
\end{gather}

\begin{theo}\label{T:GSr2SplitG3}
Every three wholly even G\"{o}pel systems $(A_k P)$, $k=1$, $2$, $3$, 
of $4$ characteristics, as defined in Theorem~\ref{T:GSr2StructG3},
give rise to three monomials in theta constants:
\begin{subequations}\label{Psi2Phi2G3}
\begin{align}
&\psi_2 (A_1P) = \theta[\emptyset] \prod_{[\I_0] \in (A_1P)} \theta[\I_0],\\
&\phi_2(A_k P) = \prod_{[\I_0] \in (A_k P)} \theta[\I_0],\quad k=2,\ 3.
\end{align}
\end{subequations}
Then 
\begin{enumerate}
\renewcommand{\labelenumi}{\arabic{enumi})}

\item $\psi_2(A_1P) \phi_2(A_k P)$, $k=2$, $3$, 
vanish to the first order on the hyperelliptic locus of $\mathfrak{S}_3$, and
$\rho\big(\partial_\tau \big(\psi_2 (A_1P) \phi_2 (A_k P) \big) \big)$
are vector-valued modular forms of weight $6$; moreover
\begin{gather*}\label{DPsi2Phi2G3}
\rho\big(\partial_\tau \big(\psi_2 (A_1P) \phi_2 (A_2 P) \big) \big) = 
\rho\big(\partial_\tau \big(\psi_2 (A_1P) \phi_2 (A_3 P) \big) \big)
= \rho\big(\partial_\tau \chi_4 \big),
\end{gather*}
where $\rho\big(\partial_\tau \chi_4 \big)$ is defined by \eqref{DChi4}.

\item   $\psi_2(A_1P)$ vanishes to the first order 
on the hyperelliptic locus of $\mathfrak{S}_3$, 
and $\rho$-images of $\phi_2(A_2 P)$ and $\phi_2(A_3 P)$ are equal:
$$\rho\big(\phi_2(A_2 P)\big) = \rho \big( \phi_2(A_3 P)\big);$$

\item $h_0(A_1,A_k,(P)) = \psi_2(A_1 P)/ \phi_2(A_k P)$, $k=2$, $3$,  
vanish on the hyperelliptic locus of $\mathfrak{S}_3$
to the first order, and 
\begin{multline}\label{0Psi2Phi2G3}
\rho \big( \partial_\tau h_0(A_1,A_2,(P)) \big)  = 
\rho \big( \partial_\tau h_0(A_1,A_3,(P)) \big) \\
= - \frac{1}{4\imath \pi} [i_1 j_1] [i_2 j_2] [i_3j_3] [i_4j_4]\, \omega \hat{S}(\emptyset) \omega^t,
\end{multline}
where all factors $[ij]=e_i-e_j$ are normally ordered.
\end{enumerate}
\end{theo}

Note, that $h_0(A,B,(P))$ defined by \eqref{h0G3} 
is the same function as $h_0(A,B,(P)) = \psi_2(A P)/ \phi_2(B P)$ with
$\psi_2$, $\phi_2$ defined by \eqref{Psi2Phi2G3}

\begin{cor}\label{T:GSr2Invar}
In genus $3$,
the sum of $h_0 (A, B, (P))$ over all $105$
pairs of G\"{o}pel systems $(A P)$, $(B P)$ such that
$(AP)$ is of the type $1[\I_2]+3[\I_0]$, 
and $(B P)$ of the type $4[\I_0]$  is a modular function of weight $0$, which 
vanishes to the first order on the hyperelliptic locus of $\mathfrak{S}_3$,
and
\begin{gather*}
\sum_{105 (P)} \partial_\tau h_0 (A, B, (P)) = 
- \frac{1}{4\imath \pi} I_1(e)\, \omega \hat{S}(\emptyset) \omega^t,
\end{gather*}
where $I_1$ is a quasi-invariant of weight $1$ and degree $4$ of a binary $8$-form:
\begin{gather}\label{I1DefG3}
I_1(e) = \sum_{105}  [i_1 j_1] [i_2 j_2] [i_3j_3] [i_4j_4],
\end{gather}
all factors $[ij]=e_i-e_j$ are normally ordered.
\end{cor}
We call $\I_1$ a quasi-invariant since it is invariant under unimodular transformations 
of a binary $8$-form, but no symmetric with respect to permutations of the roots $\{e_i\}_{i=0}^7$.

\section{Extension of $\rho$  in genus 4}
In genus $4$, among $2^8=256$ characteristics we have $10$ singular even characteristics
of multiplicity $2$, namely:  $[\I_2]=[\{k\}]$, $k=0$, \ldots $9$, 
then  $120$ non-singular odd characteristics, that is of the type $[\I_1]$, and 
 $126$ non-singular even characteristics, that is of the type $[\I_0]$.
If characteristics of branch points are chosen as in subsection~\ref{ss:CharHyper}, then 
\begin{gather*}
[\{0\}]=\big({}^{1111}_{0101}\big),\ \  [\{1\}]=\big({}^{0111}_{0101}\big),\ \
[\{2\}]=\big({}^{0111}_{1101}\big), \ \  [\{3\}]=\big({}^{1011}_{1101}\big),\ \
[\{4\}]=\big({}^{1011}_{1001}\big),\\
[\{5\}]=\big({}^{1101}_{1001}\big),\ \ [\{6\}]=\big({}^{1101}_{1011}\big),\ \
[\{7\}]=\big({}^{1110}_{1011}\big),\ \  [\{8\}]=\big({}^{1110}_{1010}\big),\ \ 
[\{9\}]=\big({}^{1111}_{1010}\big).
\end{gather*}
We denote by $\I_{All}$ the set of indices $\{i\}_{i=0}^9$ of all roots of a binary $10$-form.

Again we use G\"{o}pel systems to construct monomials in even theta constants, vanishing 
on the hyperelliptic locus of $\mathfrak{S}_4$. 
We find that vanishing monomials of weight $8$ contain two singular even characteristics,
and so vanish to the second order.  $\rho$-Images of
their second derivatives depend only on two roots which define the two singular even characteristics.

Similary to $\chi_{18}$ in genus $3$, we construct a modular form
 $\chi_{68}$, which vanishes to the $10$-th order, and find the $\rho$-image 
 of the tensor of its $10$-th derivatives.

\begin{lemma}\label{CMFI2G4}
On the hyperelliptic locus of $\mathfrak{S}_4$ 
theta constants $\theta[\{k\}]$, $k=0$, $1$, \ldots, $9$,  
with singular even characteristics vanish 
to the first order with respect to $\tau$. The matrices  
$\partial_\tau \theta[\{k\}]$ obey the transformation law \eqref{I2ModTrans}, and
\begin{gather}\label{Chi5s2G4}
\partial_\tau \theta[\{k\}]
 = - \frac{\epsilon}{4\imath \pi}  \Big(\frac{\det \omega}{\pi^4}\Big)^{1/2} \Delta(\I_{\text{All}}^{(k)})^{1/4}  
 \omega \hat{S}([\{k\}]) \omega^t,
\end{gather}
where $\epsilon^8=1$, $\omega$ is a not nomalized period matrix, $\Delta(\I_{\text{All}}^{(k)})$ denotes 
the Vandermonde determinant in $\I_{\text{All}} \backslash \{k\}$, and
$$\hat{S}[\{k\}] =  \begin{pmatrix} 0 & 0 & -1 & e_k \\ 0 & 2 & -e_k & -e_k^2 \\ 
 -1 & -e_k & 2e_k^2 & 0 \\ e_k & - e_k^2 & 0 & 0 \end{pmatrix}. $$
\end{lemma}
\begin{proof}
This Lemma is a particular case of Lemma~\ref{CMFI2}, when genus is $g=4$.
\end{proof}

\subsection{All even characteristics}
Similarly to $\chi_{18}$ in genus $3$, we define
\begin{equation}\label{Chi68Gen}
\chi_{68} = \prod_{\text{all even }[\I]} \theta[\I].
\end{equation}

\begin{theo}
On the hyperelliptic locus of $\mathfrak{S}_4$, $\chi_{68}$ acquires the form
\begin{equation}\label{Chi18}
\chi_{68} = \Big(\prod_{a=1}^{10} \theta[\I_2^{(a)}]\Big)
\Big( \prod_{a=1}^{126} \theta[\I_0^{(a)}] \Big),
\end{equation}
and vanishes to the $10$-th order. Moreover,
 $\partial^{\otimes 10}_{\tau} \chi_{68}$
 is a vector-valued modular form of weight $88$,  and
\begin{gather}\label{DChi68Dtau}
\partial^{\otimes 10}_\tau \chi_{68} = 
\frac{1}{(4\imath \pi)^{10}} \Big(\frac{\det \omega}{\pi^4}\Big)^{68}  
 \Delta^{16} \omega^{\otimes 10} \Big(\SymOtimes \prod_{k=0}^{9} 
\hat{S}(\{k\})  \Big)(\omega^t)^{\otimes 10}.
\end{gather}
\end{theo}
\begin{proof}
The lowest non-vanishing derivative of $\chi_{68}$ is of order $10$, namely
 \begin{gather*}
\partial^{\otimes 10}_\tau \chi_{68} =  \Big(\SymOtimes \prod_{a=1}^{10} 
\big(\partial_\tau  \theta[\I_2^{(a)}]\big) \Big) \prod_{a=1}^{126} \theta[\I_0^{(a)}].
\end{gather*}
By direct computations, based on
Lemma~\ref{CMFI2G4} and the First Thomae theorem, \eqref{DChi68Dtau}
is obtained. 
Evidently, all components of $\partial^{\otimes 10}_\tau \chi_{68}$
are modular forms. 
The scalar multiple on the right hand side of  \eqref{DChi68Dtau} is a binary invariant, 
 and entries of the symmetrized tensor product 
are symmetric functions in roots $\{e_i\}_{i=0}^9$ of a binary form of degree $10$.
\end{proof}

\subsection{G\"{o}pel systems of $16$ characteristics}
In genus $4$ there are 2295  G\"{o}pel groups of rank $4$, and each has one wholly even G\"{o}pel system.
Among them $945$ G\"{o}pel systems have the type $16[\I_0]$,  and $1350$  have the type $2[\I_2]+14[\I_0]$.
The latter are of our particular interest.

Note that there are no G\"{o}pel systems corresponding to 
rank $4$ G\"{o}pel groups with only one singular even
characteristic.

\begin{theo}\label{T:GSr4Struct}
Let $(A P)$ be a wholly even G\"{o}pel system of $16$ characteristics in genus $4$
with two singular characteristics, say $[\{\kappa_1\}]$, $[\{\kappa_2\}]$. 
Then $(A P)$ is formed by characteristics
\begin{gather}\label{GSr4ch2I2p14I0G4}
\begin{split}
[\{\kappa_1\}],  \qquad  &[\{\kappa_2\}],\\ 
[\{\kappa_1,i_1,i_2,i_3,i_4\}],\qquad
&[\{\kappa_2,i_1,i_2,i_3,i_4\}],\\
[\{\kappa_1,i_1,i_2,j_1,j_2\}],\qquad
&[\{\kappa_2,i_1,i_2,j_1,j_2\}],\\
[\{\kappa_1,i_1,i_2,j_3,j_4\}],\qquad
&[\{\kappa_2,i_1,i_2,j_3,j_4\}],\\
[\{\kappa_1,i_1,i_3,j_1,j_3\}],\qquad
&[\{\kappa_2,i_1,i_3,j_1,j_3\}], \\
[\{\kappa_1,i_1,i_3,j_2,j_4\}],\qquad
&[\{\kappa_2,i_1,i_3,j_2,j_4\}], \\
[\{\kappa_1,i_1,i_4,j_1,j_4\}],\qquad
&[\{\kappa_2,i_1,i_4,j_1,j_4\}], \\
[\{\kappa_1,i_1,i_4,j_2,j_3\}],\qquad
&[\{\kappa_2,i_1,i_4,j_2,j_3\}].
\end{split}
\end{gather}
where $\{i_1,i_2,i_3,i_4,j_1,j_2,j_3,j_4\}$ are chosen from 
$8$ indices $\I_{All}\backslash \{\kappa_1,\kappa_2\}$.
\end{theo}

\begin{proof} is based on straightforward computations.

Indeed, there are exist $1350$  wholly even G\"{o}pel systems of $16$ characteristics
of the form \eqref{GSr4ch2I2p14I0G4}.
Two singular characteristics $[\{\kappa_1\}]$,  $[\{\kappa_2\}]$
can be chosen from $10$ indices in $45$ ways. 
The remaining 8 indices $\I_{\text{All}}\backslash \{\kappa_1,\kappa_2\}$ 
split into two equal size parts $\{i\}\cup\{j\} = \{i_1,i_2,i_3,i_4\}\cup\{j_1,j_2,j_3,j_4\}$ in $35$ ways,
 the order is not essential. 
Each part of four indices is divided into two pairs in $3$ ways, for example:
$\{i_1,i_2\} \cup \{i_3,i_4\}$, $\{i_1,i_3\}\cup \{i_2,i_4\}$, $\{i_1,i_4\}\cup \{i_2,i_3\}$ are
partitions of the part $\{i\}$. There are $6$ ways to combine a partition of the part $\{i\}$ 
with a partition of the part $\{j\}$. Thus, we obtain $210$ combinations like
$\{i_1,i_2\} \cup \{i_3,i_4\}$, $\{j_1,j_2\} \cup \{j_3,j_4\}$. Such $210$ combinations
split into $30$ collections of $7$, according to \eqref{GSr4ch2I2p14I0G4}
where the first column defined the first partition in a combination. 
\end{proof}

\begin{theo}
Every wholly even G\"{o}pel system $(AP)$ of $16$ characteristics 
as defined in Theorem~\ref{T:GSr4Struct},
with two singular characteristics, say $[\{\kappa_1\}]$, $[\{\kappa_2\}]$,
gives rise to a monomial in theta constants
\begin{gather}\label{Mu8G4}
\mu_8 (AP) = \theta[\{\kappa_1\}]  \theta[\{\kappa_2\}] \prod_{\I_0 \in (AP)} \theta[\I_0]. 
\end{gather}
The monomials $\mu_8$ vanish to the second order on the hyperelliptic locus of $\mathfrak{S}_4$,  
and  $\partial^{\otimes 2}_\tau \mu_8(AP)$ depends only on the roots with indices $\kappa_1$ and $\kappa_2$:
\begin{gather}\label{DMu8G4}
\partial_\tau^{\otimes 2} \mu_8 (\kappa_1,\kappa_2) =
\frac{1}{(4\imath \pi)^2} \Big(\frac{\det \omega}{\pi^4}\Big)^{8} 
\frac{\Delta^2}{[\kappa_1\kappa_2]^2} \, \omega^{\otimes 2} 
\SymOtimes \big\{\hat{S}(\{\kappa_1\}), \hat{S}(\{\kappa_2\})\big\} (\omega^t)^{\otimes 2},
\end{gather}
where $[\kappa_1\kappa_2]=\Ord(e_{\kappa_1}-e_{\kappa_2})$, and 
$\SymOtimes \big\{ \dots \big\}$ 
denotes a symmetrized tensor product.
\end{theo}
This result follows from Lemma~\ref{CMFI2G4} and the first Thomae formula.


\end{document}